\documentclass[11pt,reqno]{amsart}

\usepackage{amsmath,amssymb,amsthm}
\usepackage{thmtools,enumerate}
\usepackage{mathtools}

\usepackage[german,english]{babel}
\usepackage[autostyle]{csquotes}

\usepackage[all]{xy}
\usepackage{pstricks}

\usepackage{tikz-cd}
\usetikzlibrary{babel}

\usepackage{hyperref}
\hypersetup{
	colorlinks=true,
	linkcolor=red,
	citecolor=blue}

\theoremstyle{plain}

\newtheorem{thm}{Theorem}[section]
\newtheorem{prop}[thm]{Proposition}
\newtheorem{lem}[thm]{Lemma}
\newtheorem{cor}[thm]{Corollary}

\theoremstyle{definition}

\newtheorem{dfn}[thm]{Definition}
\newtheorem{rem}[thm]{Remark}
\newtheorem{exa}[thm]{Example}

\newcommand{\Z}{\mathbb{Z}}

\newcommand{\Q}{\mathbb{Q}}
\newcommand{\R}{\mathbb{R}}
\newcommand{\C}{\mathbb{C}}
\newcommand{\OO}{\mathcal{O}}

\DeclareMathOperator{\codim}{codim}
\DeclareMathOperator{\mult}{mult}

\DeclareMathOperator{\Supp}{Supp}
\DeclareMathOperator{\Sym}{Sym}

\DeclareMathOperator{\Ext}{Ext}

\begin{document}
	
	\title{Rigid currents in birational geometry}
	
	\author[V.\ Lazi\'c]{Vladimir Lazi\'c}
	\address{Fachrichtung Mathematik, Campus, Geb\"aude E2.4, Universit\"at des Saarlandes, 66123 Saarbr\"ucken, Germany}
	\email{lazic@math.uni-sb.de}
		
	\author[Z.\ Xie]{Zhixin Xie}
	\address{Institut \'Elie Cartan de Lorraine, Universit\'e de Lorraine, 54506 Nancy, France}
	\email{zhixin.xie@univ-lorraine.fr}

	\thanks{Lazi\'c gratefully acknowledges support by the Deutsche Forschungsgemeinschaft (DFG, German Research
Foundation) -- Project-ID 286237555 -- TRR 195. We thank Tien-Cuong Dinh for explaining Proposition \ref{pro:infinitelymanyminimal} and for many useful comments, to Matei Toma for pointing out Example \ref{exa:nonkahler}, and to Hsueh-Yung Lin and Thomas Peternell for valuable conversations on the topic of this paper.
	}
	
	\begin{abstract}
		A rigid current on a compact complex manifold is a closed positive current whose cohomology class contains only one closed positive current. Rigid currents occur in complex dynamics, algebraic and differential geometry. The goals of the present paper are: (a) to give a systematic treatment of rigid currents, (b) to demonstrate how they appear within the Minimal Model Program, and (c) to give many new examples of rigid currents.
	\end{abstract}

	\maketitle
	
	%Table of Contents
	\begingroup
		\hypersetup{linkcolor=black}
		\setcounter{tocdepth}{1}
		\tableofcontents
	\endgroup
	
\section{Introduction}

Rigid cohomology classes on a compact complex manifold $X$ are those which contain precisely one closed positive currents; such currents are then called \emph{rigid currents}. Rigid currents have occurred in different contexts in complex dynamics, algebraic geometry and differential geometry: we review some of these in Section \ref{sec:rigid}. Only recently have they been given a name in \cite{SSV23}, where they were investigated in the context of hyperk\"ahler manifolds.

There are several goals of the present paper. The first is to give a systematic treatment of rigid currents, especially with respect to surjective morphisms. We also analyse in detail their relationship with currents with minimal singularities. The second goal is to demonstrate how they appear within the Minimal Model Program (MMP). And the third is to give many new examples of rigid currents.

To give some context, let $(X,\Delta)$ be a projective log canonical pair of dimension $n$, and assume that we know the existence of good minimal models in dimension $n-1$. The question that motivates us is the following: assume that we know that $\kappa_\iota(X,K_X+\Delta)\geq0$;\footnote{Here, $\kappa_{\iota}$ denotes the invariant Iitaka dimension, see Section \ref{sec:prelim}.} how can we conclude that $(X,\Delta)$ has a good minimal model? When $\kappa_\iota(X,K_X+\Delta)>0$, then we now know that $(X,\Delta)$ has a good minimal model by \cite[Theorem 1.3]{Laz24}, hence the main remaining case is when $\kappa_\iota(X,K_X+\Delta)=0$. In that case, if a good minimal model of $(X,\Delta)$ exists, then $K_X+\Delta$ defines a rigid current by Lemma \ref{lem:rigidMMP}, after pulling back to a resolution of $X$. Thus, \emph{a posteriori} we know the rigidity of $K_X+\Delta$. \emph{A priori}, if one knew the rigidity of $K_X+\Delta$, there is some evidence that this fact would be an important ingredient in a proof that $(X,\Delta)$ has a good minimal model.

This defines the setup of our paper: if we are given a projective log canonical pair $(X,\Delta)$ such that $\kappa_\iota(X,K_X+\Delta)=0$, we would like to show that $K_X+\Delta$ defines a rigid current after pulling back to a resolution of $X$. We are not yet able to prove that in this generality. However, we will show in this paper that \emph{each component} of the corresponding effective divisor defines a rigid current on a sufficiently high resolution of $X$.

The following are the main results of the paper. The first deals with pairs with klt singularities.

\begin{thm}\label{thm:main1}
    Assume the existence of good models for projective log canonical pairs in dimension $n-1$.
    
    Let $(X,\Delta)$ be a projective klt pair of dimension $n$ such that $\kappa_\iota(X,K_X+\Delta )=0$, and let $D\geq0$ be the unique $\R$-divisor such that $K_X+\Delta\sim_\R D$. Let $f\colon Y\to X$ be a log resolution of $(X,\Delta+D)$. Then for each desingularisation $\pi\colon Z\to X$ which factors through $f$, each component of $\pi^*D$ defines a rigid current.
  \end{thm}

When the pairs are log canonical, we can prove a similar result, albeit under somewhat stronger assumptions.

\begin{thm}\label{thm:main2}
    Assume the Nonvanishing conjecture in dimension $n$ and the existence of good models for projective log canonical pairs in dimension $n-1$.
    
    Let $(X,\Delta)$ be a projective log canonical pair of dimension $n$ such that $\kappa_\iota(X,K_X+\Delta )=0$, and let $D\geq0$ be the unique $\R$-divisor such that $K_X+\Delta\sim_\R D$. Let $f\colon Y\to X$ be a log resolution of $(X,\Delta+D)$. Then for each desingularisation $\pi\colon Z\to X$ which factors through $f$, each component of $\pi^*D$ defines a rigid current.
  \end{thm}
  
In particular, in the notation of the theorems, for each component $\Gamma$ of $\pi^*D$, we have that \emph{$\Gamma$ is the only effective divisor numerically equivalent to $\Gamma$}. This is a very special property of these components: indeed, as Example \ref{exa:lehmann} shows, there exist a smooth projective variety $Z$ and a smooth prime divisor $P$ on $Z$ such that $\kappa(Z,P)=0$, but there exists another effective divisor $P'$ numerically equivalent to $P$. However, note that the rigidity of currents of integration, as proved in Theorems \ref{thm:main1} and \ref{thm:main2}, is a much stronger property than saying that the corresponding class contains only one current of integration of an effective $\R$-divisor, see Remark \ref{rem:rigidvsnumerical}.
	
\section{Preliminaries}\label{sec:prelim}

Throughout the paper we work over $\C$, and all manifolds are connected. We write $D\geq0$ for an effective $\R$-divisor $D$ on a normal variety $X$. A \emph{birational contraction} is a birational map whose inverse does not contract any divisors. If $f\colon X\to Y$ is a surjective morphism of compact complex manifolds or of normal varieties, and if $D$ is an $\R$-divisor on $X$, then $D$ is \emph{$f$-exceptional} if $\codim_Y f(\Supp D)\geq2$. We use the convention that $d^c=\frac{1}{2\pi i}(\partial-\overline{\partial})$, so that $dd^c=\frac{i}{\pi}\partial\overline{\partial}$. 

\subsection{Pairs}

The standard reference for the definitions and basic results on the singularities of pairs and the Minimal Model Program is \cite{KM98}. A \emph{pair} $(X,\Delta)$ in this paper always has a boundary $\Delta$ which is an effective $\R$-divisor. Many results on the existence of minimal models for log canonical pairs are now known, see \cite{LT22}.

\subsection{Minimal and good models}

Let $(X,\Delta)$ be a log canonical pair and let $Y$ be a normal variety. A birational contraction $f\colon X\dashrightarrow Y$ is a \emph{minimal model of $(X,\Delta)$} if $K_Y+f_*\Delta$ is $\R$-Cartier and nef, and if there exists a resolution of indeterminacies $(p,q)\colon W\to X\times Y$ of the map $f$ such that $p^*(K_X+\Delta)\sim_\R q^*(K_Y+f_*\Delta)+E$, where $E\geq0$ is a $q$-exceptional $\R$-divisor which contains the whole $q$-exceptional locus in its support. If $K_Y+f_*\Delta$ is additionally semiample, them $f$ is a \emph{good minimal model}, or simply a \emph{good model}, of $(X,\Delta)$.

\subsection{Invariant Iitaka dimension}

We use the \emph{invariant Iitaka dimension} of a pseudoeffective $\R$-Cartier $\R$-divisor $D$ on a normal projective variety $X$, denoted by $\kappa_\iota(X,D)$. It shares many of the good properties with the Iitaka dimension for $\Q$-divisors \cite[Section 2.5]{Fuj17} that will be used in this paper without explicit mention. In particular, the invariant Iitaka dimension of a $\Q$-divisor is the same as its usual Iitaka dimension by \cite[Proposition 2.5.9]{Fuj17}. If the divisor $D$ has rational coefficients or if $D\geq0$, its Iitaka dimension is denoted by $\kappa(X,D)$. 

We need two properties which will be used without explicit mention, see \cite[Section 2]{Laz24}. First, if $D$ is an $\R$-Cartier $\R$-divisor on a normal projective variety $X$, if $f\colon Y\to X$ is a birational morphism from a normal projective variety $Y$, and if $E$ is an effective $f$-exceptional divisor on $Y$, then
$$\kappa_\iota(X,D)=\kappa_\iota(Y,f^*D+E).$$
Second, if $D_1$ and $D_2$ are effective $\R$-Cartier $\R$-divisors on a normal projective variety $X$ such that $\Supp D_1=\Supp D_2$, then $\kappa_\iota(X,D_1)=\kappa_\iota(X,D_2)$.

\subsection{Positive currents}

If $X$ is a complex manifold, we denote its Bott-Chern $(1,1)$-cohomology space by $H^{1,1}_\mathrm{BC}(X,\C)$, and we denote by $H^{1,1}_\mathrm{BC}(X,\R)$ the space of its real points. If additionally $X$ is compact and K\"ahler, then $H^{1,1}_\mathrm{BC}(X,\C)$ is isomorphic to the Dolbeault cohomology group $H^{1,1}(X,\C)$.

If $T$ is a closed $(1,1)$-current on a complex manifold $X$, we denote by $\{T\}$ its class in $H^{1,1}_\mathrm{BC}(X,\C)$; if $T$ is real, then $\{T\}\in H^{1,1}_\mathrm{BC}(X,\R)$. If $T$ is a representative of a class $\alpha\in H^{1,1}_\mathrm{BC}(X,\C)$, we write $T\in\alpha$; if $T'\in\alpha$ is another representative, we also write $T\equiv T'$.

A $(1,1)$-current $T$ on $X$ is \emph{positive}, and we write $T\geq0$, if $T(\varphi)$ is a positive measure for every smooth $(n-1,n-1)$-form of type 
$$\varphi=(i\alpha_1\wedge \overline\alpha_1)\wedge \ldots\wedge(i\alpha_{n-1}\wedge \overline\alpha_{n-1}),$$
where $\alpha_i$ are $(1,0)$-forms and $n=\dim X$. A positive $(1,1)$-current is always real. If $T$ and $T'$ are two $(1,1)$-currents on $X$, we write $T\geq T'$ if $T-T'\geq0$. A cohomology class $\alpha\in H^{1,1}(X,\R)$ is \emph{pseudoeffective} if it contains a closed positive current.

If $D$ is an irreducible analytic subset of pure codimension $1$ in $X$, then we denote by $[D]$ (or simply also by $D$ if there is no danger of confusion) the \emph{current of integration} on the regular part of $D$, which is a positive closed $(1,1)$-current. If we have an $\R$-divisor $D=\delta_1D_1+\dots+\delta_rD_r$ on $X$, then we call the corresponding current $[D]:=\delta_1[D_1]+\dots+\delta_r[D_r]$ the \emph{current of integration on $D$}.

If $f\colon Y\to X$ is a surjective holomorphic map between compact complex manifolds and if $T$ is a closed positive $(1,1)$-current on $X$, then one can easily define its pullback $f^*T$ to $Y$ such that $\{f^*T\}=f^*\{T\}$, see \cite[2.2.3]{Bou04}. If $D$ is an $\R$-divisor on $X$, then $f^*[D]=[f^*D]$; this follows from the Lelong--Poincar\'e equation.
If $G$ is an $\R$-divisor on $Y$, then $f_*[G]=[f_*G]$, see for instance \cite[Proposition 4.2.78]{BM19}. This will be relevant in Remark \ref{rem:pushforward}.

\subsection{Plurisubharmonic functions}

We refer to \cite{Dem12} for the definition and general properties of \emph{plurisubharmonic} or \emph{psh} functions on a complex manifold $X$. Here we only collect several properties that we use often in this paper.

For psh functions $u$ and $v$ on a complex manifold $X$, the functions $u+v$ and $\max\{u,v\}$ are also psh. The pullback of a psh function by a holomorphic map is again psh. A closed $(1,1)$-current $T$ on $X$ is positive if and only if it can be written locally as $T=dd^c \varphi$ for a psh function $\varphi$.

Psh functions are locally bounded from above on $X$; if additionally $X$ is compact, then any psh function on $X$ is constant. A more suitable notion on compact complex manifolds is that of \emph{quasi-plurisubharmonic} or \emph{quasi-psh} functions. A function $\varphi\colon X\to[-\infty,+\infty)$ on a complex manifold $X$ is quasi-psh if it is locally equal to the sum of a psh function and of a smooth function; equivalently, $\varphi$ is quasi-psh if it is integrable and upper semicontinuous, and there exists a smooth $(1,1)$-form $\theta$ on $X$ such that $\theta+dd^c\varphi\geq 0$ in the sense of currents: then we say that $\varphi$ is \emph{$\theta$-psh}.

If $\theta$ is a smooth closed real $(1,1)$-form on $X$ and if $T$ is a closed positive $(1,1)$-current in $\{\theta\}$, then there exists a locally integrable function $\varphi$ such that $T=\theta+dd^c\varphi$. In other words, $\varphi$ is $\theta$-psh. If $X$ is additionally compact, then $\varphi$ is unique up to an additive constant.

\subsection{Lelong numbers}

We refer to \cite{Dem12} for the definition and general properties of \emph{Lelong numbers} of psh functions on a complex manifold $X$. Here we recall several properties that we use often in this paper.

For psh functions $u$ and $v$ on $X$ and for each point $x\in X$ we have
$$ \nu(u+v,x)=\nu(u,x)+\nu(v,x). $$
If $T$ is a closed positive $(1,1)$-current, then locally around a point $x\in X$ we can write $T=dd^c\varphi$ for a psh function $\varphi$, and we define the Lelong number of $T$ at $x$ as $\nu(T,x):=\nu(\varphi,x)$; this does not depend on the choice of $\varphi$. If $Y$ is an analytic subset of $X$ and if $x\in X$, then a result of Thie states that $\nu(Y,x)$ is equal to the multiplicity of $Y$ at $x$. For a closed positive $(1,1)$-current $T$ and for any analytic subset $Y$ of $X$ we may define the \emph{generic Lelong number of $T$ along $Y$} as 
$$\nu(T,Y):=\inf_{x\in Y}\nu(T,x),$$
which is equal to $\nu(T,x)$ for a very general point $x\in Y$ by a theorem of \cite{Siu74}.

If $T$ is a closed positive $(1,1)$-current on $X$, then by \cite{Siu74} there exist at most countably many codimension $1$ analytic subsets $D_k$ such that $T$ has the \emph{Siu decomposition}
$$T=\sum\nu(T,D_k)D_k+R,$$
where $R$ is a closed positive $(1,1)$-current such that $\nu(R,\Gamma)=0$ for each codimension $1$ analytic subset of $X$. In this paper we call $\sum\nu(T,D_k)D_k$ the \emph{divisorial part} and $R$ the \emph{residual part} of (the Siu decomposition of) $T$.

\subsection{Nakayama--Zariski functions and Boucksom--Zariski functions}

Let $X$ be a $\Q$-factorial projective variety and let $\Gamma$ be a prime divisor on $X$. Nakayama \cite{Nak04} defined \emph{$\sigma_\Gamma$-functions} on the pseudoeffective cone on $X$. Namely, if $D$ be a big $\R$-divisor on $X$, set
    \[\sigma_\Gamma (D) := \inf \{ \mult_\Gamma \Delta \mid 0 \leq\Delta \sim_\R D \};\]
    and if $D$ is a pseudoeffective $\R$-divisor on $X$, we pick an ample $\R$-divisor $A$ on $X$ and define
    \[\sigma_\Gamma (D) := \lim_{\varepsilon\downarrow 0} \sigma_\Gamma (D+\varepsilon A);\]
   this does not depend on the choice of $A$ and is compatible with the definition above for big divisors. Moreover, $\sigma_\Gamma(D)$ only depends on the numerical class of $D$, hence $\sigma_\Gamma$ is well-defined on the pseudoeffective cone of $X$. Nakayama originally defined these functions when $X$ is smooth, but the definition works well in the $\Q$-factorial setting, see for instance \cite[Lemma 2.12]{LX23}.
   
If $X$ is a compact complex manifold and if $\Gamma$ is an analytic prime divisor on $X$, Boucksom \cite{Bou04} defined \emph{$\nu(\cdot,\Gamma)$-functions} on the cone of pseudoeffective classes in $H^{1,1}_\mathrm{BC}(X,\R)$, and he showed that they coincide with Nakayama's $\sigma_\Gamma$-functions when one considers algebraic classes. To avoid possible confusion with Lelong numbers, we will denote these Boucksom's functions also by $\sigma_\Gamma$. Here we sketch the construction when $X$ is a compact K\"ahler manifold, which suffices for the purposes of this paper. Let $\alpha$ be a pseudoeffective class in $H^{1,1}(X,\R)$. After fixing a reference K\"ahler form $\omega$, and if $T_{\min,\varepsilon}$ is a current with minimal singularities in the class $\alpha+\varepsilon\{\omega\}$ for a positive real number $\varepsilon$ (see Section \ref{sec:minimal}), set
$$\sigma_\Gamma(\alpha):=\inf_{x\in\Gamma}\sup_{\varepsilon>0}\nu(T_{\min,\varepsilon},x);$$
this does not depend on the choice of $\omega$ and one has $\sigma_\Gamma(\alpha)=\nu(T_{\min},\Gamma)$ when $\alpha$ is a big class and $T_{\min}\in \alpha$ is a current with minimal singularities. Even though the notation is slightly different, it is easy to see that this definition is equivalent to that from \cite{Bou04}.

The following easy lemma will be crucial in this paper, see \cite[Proposition 3.6(i)]{Bou04} and the equation \eqref{eq:Tmin} below.

\begin{lem}\label{lem:sigmanu}
Let $X$ be a compact complex manifold and let $\Gamma$ be an analytic prime divisor on $X$. If $\alpha$ is a pseudoeffective class in $H^{1,1}_\mathrm{BC}(X,\R)$, then $\sigma_\Gamma(\alpha)\leq\nu(T,\Gamma)$ for every $T\in\alpha$.
\end{lem}

The following well-known lemma relates $\sigma_\Gamma$-functions and the MMP; for the most general version, see for instance \cite[Lemma 2.14]{LX23}.

\begin{lem}\label{lem:sigmaMMP}
Let $(X,\Delta)$ be a projective $\Q$-factorial log canonical pair and let $\Gamma$ be a prime divisor such that $\sigma_\Gamma(K_X+\Delta)=0$. If $\varphi\colon (X,\Delta)\dashrightarrow(Y,\Delta_Y)$ is a minimal model of $(X,\Delta)$, then $\Gamma$ is not contracted by $\varphi$.
\end{lem}

\section{Rigid currents}\label{sec:rigid}

Let $X$ be a compact complex manifold. A class $\alpha\in H^{1,1}_\mathrm{BC}(X,\R)$ is \emph{rigid} if it contains exactly one closed positive $(1,1)$-current; thus, $\alpha$ is necessarily pseudoeffective. A closed positive $(1,1)$-current $T$ on $X$ is \emph{rigid} if its class $\{T\}\in H^{1,1}_\mathrm{BC}(X,\R)$ is rigid. Another way of expressing this is as follows: following \cite{SSV23}, for a pseudoeffective class $\alpha\in H^{1,1}_\mathrm{BC}(X,\R)$ we denote
$$\mathcal C_\alpha:=\{T\in \alpha\mid T\geq0\}.$$
The $\alpha$ is rigid if $\mathcal C_\alpha$ contains exactly one element. This notation will be particularly useful in Section \ref{sec:fibrations}.

\begin{rem}\label{rem:smallercurrent}
Let $X$ be a compact complex manifold and let $T$ be a rigid $(1,1)$-current on $X$. If $T'$ is a closed positive $(1,1)$-current on $X$ such that $T'\leq T$, then $T'$ is also rigid. Indeed, if $S\in\{T'\}$ were another closed positive $(1,1)$-current, then $S+(T-T')$ would be a closed positive $(1,1)$-current in $\{T\}$ different than $T$.
\end{rem}

Rigid currents have appeared in different contexts in the literature, but it seems they were first defined explicitly in the recent paper \cite{SSV23}. We review next some of the examples of rigid currents in different contexts.

\begin{exa}\label{exa:zero}
Let $X$ be a compact complex manifold. Then the class $\{0\}\in H^{1,1}_\mathrm{BC}(X,\R)$ is rigid: this is the first and basic, but fundamental example of rigidity of currents. Indeed, if $T\in\mathcal C_{\{0\}}$, then there exists a distribution $\varphi$ on $X$ such that $T=dd^c\varphi$. As $T\geq0$, we may identify $\varphi$ with a psh function on $X$. By compactness, $\varphi$ must be constant, hence $T=0$, as desired.
\end{exa}

\begin{exa}\label{exa:exceptional}
Let $\pi\colon Y\to X$ be a birational morphism between complex projective manifolds and let $E$ be an effective $\pi$-exceptional $\R$-divisor on $Y$. Then $E$ is rigid. We will show a more general result in Corollary \ref{cor:Tplusexceptional} below.
\end{exa}

\begin{exa}
One of the earliest examples of rigid currents occurred in complex dynamics \cite{Can01}. We refer to \cite{SSV23} for details, here we only recall the following special case of an important result from \cite{DS10}: consider an automorphism $f\colon X\to X$ of a compact K\"ahler manifold, and let $d_1(f)$ be the spectral radius of $f^*\colon H^{1,1}(X,\R)\to H^{1,1}(X,\R)$. Then we know that $d_1(f)\geq 1$ in general. If, however, $d_1(f)>1$, then it follows from \cite[Corollary 4.3.2]{DS10} that there exists a non-zero closed positive $(1,1)$-current $T$ on $X$ such that $f^*T=d_1(f)T$, and each such current $T$ is rigid.
\end{exa}

\begin{exa}
If $X$ is a K3 surface, then a nef class $\alpha\in H^{1,1}(X,\R)$ is \emph{parabolic} if $\alpha^2=0$. Then \cite[Theorem 4.3.1]{FT23} shows that each \emph{irrational} parabolic class on a projective $X$ is rigid, if the Picard rank of $X$ is at least $3$ and $X$ has no $(-2)$-curves. This was then extended in \cite[Theorem 2.1]{SSV23} to hyperk\"ahler manifolds $X$ with $b_2(X)\geq7$ (all currently known classes of hyperk\"ahler manifolds satisfy this condition) and parabolic currents on $X$, i.e.\ nef classes $\alpha\in H^{1,1}(X,\R)$ with $\alpha^{\dim X}=0$, which satisfy some additional irrationality properties.
\end{exa}

\begin{exa}
An important non-dynamical early example of rigid currents appears in \cite[Example 1.7]{DPS94}. Let $E$ be an elliptic curve and let $\mathcal E$ be a vector bundle given by the extension corresponding to a nontrivial element in $\Ext^1(\OO_E,\OO_E)\simeq H^1(E,\OO_E)\simeq\C$. Let $\sigma\colon E\to\mathbb P(\mathcal E)$ be the section of the projection $\mathbb P(\mathcal E)\to E$ corresponding to the quotient $\mathcal E\to \OO_E$, and set $C:=\sigma(E)$. Then $C$ is a nef divisor on $\mathbb P(\mathcal E)$ which is rigid as a current. In particular, this gives an example of a nef line bundle which is not hermitian semipositive.
\end{exa}

\begin{exa}\label{exa:nonkahler}
An interesting example of a rigid current on a compact non-K\"{a}hler surface is given in \cite{Tom08}. An Enoki surface $X$ is a minimal compact non-K\"ahler surface which has a unique cycle $C$ of rational curves on it in the sense of \cite[Definition 5]{Tom08}, and the image of the class $\gamma:=\{[C]\}$ under the canonical morphism $\pi\colon H^{1,1}_\mathrm{BC}(X,\R)\to H^2_\mathrm{dR}(X,\R)$ is zero. Then \cite[Theorem 10(b)]{Tom08} shows that $[C]$ is the only closed positive $d$-exact $(1,1)$-current on $X$ up to a multiplicative constant. If now $T\in\gamma$, then $T$ is $d$-exact since $\pi(\gamma)=0$, hence there exists $c\in\R$ such that $T=c[C]$. As then $c[C]\equiv[C]$, we conclude that $c=1$ by Example \ref{exa:zero}, and therefore $[C]$ is a rigid $(1,1)$-current on $X$.
\end{exa}

\begin{exa}\label{exa:lehmann}
The following example is \cite[Example 6.1]{Leh13}, which we reproduce with more details. As mentioned in the introduction, this is a possibly surprising example of a current which is not rigid: there exists a smooth projective threefold $Z$ and distinct effective divisors $P$ and $P'$ on $Z$ such that $P$ is a smooth prime divisor, $\kappa(Z,P)=0$, $\kappa(Z,P')=1$ and $P\equiv P'$. In particular, the class $\{P\}$ is not rigid.

Fix an elliptic curve $E$ and let $S:=E\times E$ with the projection morphisms $p_1$ and $p_2$. Let $F=p_1^*F'$ and $N:=p_2^*N'$, where $F'$ is a point on $E$ and $N'$ is a degree-zero non-torsion divisor on $E$. We first claim that for all $f\in\Z$ and $n\in\Z\setminus\{0\}$ we have
\begin{equation}\label{eq:infty}
\kappa(S,fF+nN)=-\infty.
\end{equation}
Indeed, assume that there exist integers $f$ and $n\neq0$ and an effective $\Q$-divisor $G$ on $S$ such that $G\sim_\Q fF+nN$, and let $F_1$ be a fibre of $p_1$ which is not contained in the support of $G+F$. Then
$$0\leq\kappa(F_1, G|_{F_1}) = \kappa(F_1, nN |_{F_1})=\kappa(E,nN')=-\infty,$$
a contradiction which shows \eqref{eq:infty}.

Now, set $\mathcal E:=\OO_S\oplus\OO_S(F+N)$ and $Z:=\mathbb{P}(\mathcal E)$ with the projection morphism $\pi\colon Z \to S$. Let $\sigma\colon S\to X$ be the section of $\pi$ corresponding to the quotient $\mathcal E\twoheadrightarrow\OO_S(F+N)$, and set $P:=\sigma(S)$. Then $P$ is a smooth prime divisor on $X$ and it is well-known that $\OO_X(P)\simeq\OO_X(1)$, see for instance the proof of \cite[Proposition V.2.6]{Har77}. Thus, by \eqref{eq:infty}, for any $m\geq 1$ we have
$$ h^0\big(Z,\OO_Z(mP)\big)=h^0\big(S,\Sym^m\mathcal E\big)=\sum_{k=0}^mh^0\big(S,\OO_S(k(F+N))\big)=1, $$
hence $\kappa(Z,P)=0$. Define $P_1:=P-\pi^*N$; then clearly $P\equiv P_1$ and by \eqref{eq:infty}, for any $m\geq 1$ we have
\begin{align*}
h^0\big(Z,\OO_Z(mP_1)\big)&=h^0\big(S,\Sym^m\mathcal E\otimes\OO_S({-}mN)\big)\\
&=\sum_{k=0}^mh^0\big(S,\OO_S(kF+(k-m)N)\big)=h^0\big(S,\OO_S(mF)\big).
\end{align*}
Therefore, $\kappa(Z,P_1)=\kappa(Z,F)=1$ and $h^0(Z,\OO_Z(P_1))>0$, and we choose $P'$ to be any effective divisor such that $P'\sim P_1$.
\end{exa}

An important example of rigid currents in the context of this paper comes from the MMP. 

\begin{lem}\label{lem:rigidMMP}
Let $(X,\Delta)$ be a projective log canonical pair which has a good model and such that $\kappa_\iota(X,K_X+\Delta)=0$. Then for each birational morphism $\pi\colon Y\to X$ such that $Y$ is smooth, the class $\{\pi^*(K_X+\Delta)\}$ is rigid.
\end{lem}

\begin{proof}
Fix a birational morphism $\pi\colon Y\to X$ such that $Y$ is smooth. Let $\varphi\colon (X,\Delta)\dashrightarrow (X',\Delta')$ be a birational contraction to a good model $(X',\Delta')$ of $(X,\Delta)$, and let $(p,q)\colon W\to X\times X'$ be a smooth resolution of indeterminacies of $\varphi$; we may assume that $p$ factors through $\pi$ and let $f\colon W\to Y$ be the induced map. Then by the Negativity lemma \cite[Lemma 3.39]{KM98} there exists an effective $q$-exceptional $\R$-divisor $E$ on $W$ such that
\begin{equation}\label{eq:8}
p^*(K_X+\Delta)\sim_\R q^*(K_{X'}+\Delta')+E.
\end{equation}
Since $\kappa_\iota(X',K_{X'}+\Delta')=0$ and $K_{X'}+\Delta'$ is semiample, we have $K_{X'}+\Delta'\sim_\R 0$, and hence $\{p^*(K_X+\Delta)\}=\{E\}$ by \eqref{eq:8}. Therefore, the class $\{p^*(K_X+\Delta)\}$ is rigid by Example \ref{exa:exceptional}. Since $p^*(K_X+\Delta)=f^*\pi^*(K_X+\Delta)$, we conclude by Corollary \ref{cor:pullback_rigid} below.
\end{proof}

\section{Rigid currents and fibrations}\label{sec:fibrations}

The following result was proved in \cite[Proposition 1.2.7]{Bou02}, and we include the proof for the benefit of the reader.

\begin{prop}\label{pro:pullback}
Let $f\colon X\to Y$ be a surjective morphism with connected fibres between compact complex manifolds, and let $\alpha\in H^{1,1}_\mathrm{BC}(Y,\R)$ be a pseudoeffective class. Then the pullback map
$$f^*\colon \mathcal C_\alpha\to\mathcal C_{f^*\alpha}$$
is bijective.
\end{prop}

\begin{proof}
Let $C\subseteq Y$ be the set of critical values of $f$, and set $Y':=Y\setminus C$ and $X':=f^{-1}(Y')$. Then the map $f|_{X'}\colon X'\to Y'$ is a holomorphic submersion. We proceed in two steps.

\medskip

\emph{Step 1.}
In this step we show that the map $f^*$ is surjective. 

Fix a smooth form $\theta\in \alpha$. If $T\in\mathcal C_{f^*\alpha}$, then there exists a quasi-psh function $\varphi$ on $X$ such that $T=f^*\theta+dd^c\varphi$. We first claim that
$$\text{the function }\varphi|_{f^{-1}(y)}\text{ is constant for every }y\in Y.$$
Indeed, fix an open subset $U\subseteq Y$ such that there exists a psh function $u$ on $U$ containing the point $y$ with $\theta|_U=dd^c u$. Then $T|_{f^{-1}(U)}=dd^c(f^*u+\varphi|_{f^{-1}(U)})$, hence the function $f^*u+\varphi|_{f^{-1}(U)}$ is psh on $f^{-1}(U)$. Since $f^{-1}(y)$ is compact, $f^*u+\varphi|_{f^{-1}(U)}$ must be constant on $f^{-1}(y)$. As $f^*u$ is constant on $f^{-1}(y)$, the claim follows.

For each $y\in Y'$ set
$$\psi(y):=\varphi(x)\quad\text{for some }x\in f^{-1}(y);$$
this is well defined by the claim above. We next claim that $\psi$ is a $\theta|_{Y'}$-psh function on $Y'$. Indeed, for each $y\in Y'$ fix an open subset $y\in U\subseteq Y'$ such that there exists a psh function $u$ on $U$ with $\theta|_U=dd^c u$. By shrinking $U$ if necessary and by the constant rank theorem we may assume that there exists an open subset $F\subseteq f^{-1}(y)$ and local coordinates $(z_1,\dots,z_n)$ on an open subset $V\subseteq f^{-1}(U)$ in $X$ such that $V\simeq U\times F$ and in those coordinates, the map $f|_V$ is just the projection $(z_1,\dots,z_n)\mapsto (z_1,\dots,z_m)$, where $n=\dim X$ and $m=\dim Y$. Then for each point $x\in F$, the set $U\times\{x\}$, viewed as a subset of $V$, is isomorphic to $U$ and the function $(u+\psi)|_U$ corresponds to $(f^*u+\varphi)|_{U\times\{x\}}$ under that isomorphism, by the definition of $\psi$. Since $(f^*u+\varphi)|_{U\times\{x\}}$ is a psh function as in the previous paragraph, the function $(u+\psi)|_U$ is also psh. This gives the claim.

Therefore, $\psi$ is a $\theta|_{Y'}$-psh function on $Y'$, and it is clearly bounded from above near $C$, since $\varphi$ is bounded from above on the whole $X$, as $X$ is compact. Therefore, $\psi$ extends uniquely to a $\theta$-psh function on $X$ by \cite[Theorem I.5.24]{Dem12}, which we also denote by $\psi$. Set $S:=\theta+dd^c\psi$; this is then a positive $(1,1)$-current on $Y$. Since $T$ and $f^*S$ coincide on $X'$, they must coincide on the whole $X$ again by \cite[Theorem I.5.24]{Dem12}, and the surjectivity of $f^*$ is proved.

\medskip

\emph{Step 2.}
In this step we show that the map $f^*$ is injective.

Assume we have two currents $T_1$ and $T_2$ in $\mathcal C_\alpha$ such that $f^*T_1=f^*T_2$, and fix a smooth form $\theta\in \alpha$. Then there exist quasi-psh functions $\psi_1$ and $\psi_2$ on $Y$ such that $T_1=\theta+dd^c\psi_1$ and $T_2=\theta+dd^c\psi_2$. Consequently, we have
$$dd^c f^*(\psi_1-\psi_2)=0.$$
Then as in the last two paragraphs of Step 1 this implies that the functions $\psi_1-\psi_2$ and $\psi_2-\psi_1$ are psh on $Y$, hence $dd^c\psi_1=dd^c\psi_2$. This shows that $T_1=T_2$, which was to be proved.
\end{proof}

An immediate corollary is the following.

\begin{cor}\label{cor:pullback_rigid}
    Let $\pi\colon Y\to X$ be a surjective morphism with connected fibres between compact complex manifolds and let $\alpha\in H^{1,1}_\mathrm{BC}(X,\R)$ be a pseudoeffective class. Then $\alpha$ is rigid if and only if $\pi^*\alpha$ is rigid.
\end{cor}

\begin{rem}\label{rem:pushforward}
    If $\pi\colon Y\to X$ is a surjective morphism between compact complex manifolds and if $\alpha\in H^{1,1}_\mathrm{BC}(Y,\R)$ is a rigid class, it is in general \emph{not} true that $\pi_*\alpha$ is rigid, even if $\pi$ is bimeromorphic. Indeed, let $X=\mathbb{P}^1\times\mathbb{P}^1$ and fix a fibre $G\subseteq X$ of the first projection to $\mathbb P^1$. Then $G$ is clearly not rigid as $\kappa(X,G)=1$. Now consider the blowup $\pi\colon Y\to X$ of $X$ at a point contained in $G$. Then as $(\pi^*G)^2=0$, we easily calculate that $\pi^*G$ is the sum of two $(-1)$-curves, and let $G'$ be the strict transform of $G$ on $Y$. Since $G'$ can be contracted to a smooth projective surface by Castelnuovo's criterion, it is rigid by Example \ref{exa:exceptional}, but $G=\pi_*G'$ is not.
\end{rem}

We have the following generalisation of Example \ref{exa:exceptional}.

\begin{cor}\label{cor:Tplusexceptional}
    Let $\pi\colon Y\to X$ be a surjective morphism with connected fibres between complex projective manifolds, let $D$ be a pseudoeffective $\R$-divisor on $X$ and let $E$ be an effective $\pi$-exceptional $\R$-divisor on $Y$. Then the map
   	$$f\colon \mathcal C_{\{D\}}\to\mathcal C_{\{\pi^*D+E\}},\quad T\mapsto\pi^*T+E$$
   	is an isomorphism. In particular, $\{D\}$ is rigid if and only if $\{\pi^*D+E\}$ is rigid.
\end{cor}

\begin{proof}
Assume $S_1$ and $S_2$ are two currents in $\{D\}$ such that $f(S_1)=f(S_2)$. Then $\pi^*S_1=\pi^*S_2$, and thus $S_1=S_2$ by Proposition \ref{pro:pullback}. This shows that $f$ is injective.

To show the surjectivity of $f$, let $T\in\mathcal C_{\{\pi^*D+E\}}$. Then for each component $\Gamma$ of $E$ we have
$$\nu(T,\Gamma)\geq\sigma_\Gamma(\pi^*D+E)\geq \mult_\Gamma E$$
by Lemma \ref{lem:sigmanu} and by \cite[Proposition III.5.7 and Lemma III.5.14]{Nak04}. Consequently, the current $T-E\in\{\pi^*D\}$ is positive by considering the Siu decomposition of $T$, hence by Proposition \ref{pro:pullback} there exists $S\in\mathcal C_{\{D\}}$ such that $T-E=\pi^*S$. Therefore, we have $T=f(S)$, as desired.
\end{proof}

\begin{rem}
It is reasonable to expect that Corollary \ref{cor:Tplusexceptional} holds also when $D$ is a closed positive $(1,1)$-current, but the proof of surjectivity above uses \cite{Nak04} crucially, which works only in the setting of algebraic classes.
\end{rem}

\section{Rigidity and currents with minimal singularities}\label{sec:minimal}

A good reference for currents with minimal singularities is \cite{Bou04}. Here we recall some of the properties we need in this paper. Let $\varphi_1$ and $\varphi_2$ be quasi-psh functions on a compact complex manifold $X$. Then we say that $\varphi_1$ is \emph{less singular} than $\varphi_2$, and write $\varphi_1\preceq\varphi_2$, if there exists a constant $C$ such that $\varphi_2\leq\varphi_1+C$. We denote by $\varphi_1\approx\varphi_2$ the induced equivalence relation, i.e.\ we say that $\varphi_1$ and $\varphi_2$ have \emph{equivalent singularities} if $\varphi_1\preceq\varphi_2\preceq\varphi_1$. 

Let $T_1$ and $T_2$ be two closed positive $(1,1)$-currents on $X$ in a fixed pseudoeffective class $\alpha\in H^{1,1}_\mathrm{BC}(X,\R)$ and let $\theta\in\alpha$ be a fixed smooth form. Then there exist $\theta$-psh functions $\varphi_1$ and $\varphi_2$ such that $T_1=\theta+dd^c\varphi_1$ and $T_2=\theta+dd^c\varphi_2$. We say that $T_1$ is \emph{less singular} than $T_2$, and write $T_1\preceq T_2$, if $\varphi_1\preceq\varphi_2$; and similarly for $T_1\approx T_2$. This does not depend on the choice of $\theta$, $\varphi_1$ and $\varphi_2$. Any two closed positive $(1,1)$-currents with equivalent singularities have the same Lelong numbers.

A compactness argument on the set of $\theta$-psh functions then shows that there exists a currents $T_{\min}\in\alpha$ such that $T_{\min}\preceq T$ for all $T\in\alpha$. We call $T_{\min}$ a \emph{positive current with minimal singularities in $\alpha$}. For each point $x\in X$ we have
\begin{equation}\label{eq:Tmin}
\nu(T_{\min},x)=\inf_{T\in\alpha}\nu(T,x).
\end{equation}

Part (a) of the following result shows that the relation $\preceq$ is compatible with pullbacks. Part (b), which is \cite[Proposition 1.12]{BEGZ}, is the consequence of (a) and of Proposition \ref{pro:pullback}.

\begin{prop}\label{pro:pullback_min_sing}
Let $f\colon Y\to X$ be a surjective morphism with connected fibres between compact complex manifolds and let $\alpha\in H^{1,1}_\mathrm{BC}(X,\R)$ be a pseudoeffective class. 
\begin{enumerate}[\normalfont (a)]
\item Let $T_1$ and $T_2$ be two $(1,1)$-currents in $\mathcal C_\alpha$. Then $T_1\preceq T_2$ if and only if $f^*T_1\preceq f^*T_2$.
\item Let $T\in\mathcal C_\alpha$. Then $T$ has minimal singularities if and only if the current $f^*T\in\mathcal C_{f^*\alpha}$ has minimal singularities.
\end{enumerate}
\end{prop}

\begin{proof}
We first show (a). Fix a smooth $(1,1)$-form $\theta\in\alpha$ and $\theta$-psh functions $\varphi_1$ and $\varphi_2$ on $X$ such that $T_1=\theta+dd^c\varphi_1$ and $T_2=\theta+dd^c\varphi_2$. If $T_1\preceq T_2$, then there exists a constant $C_1$ such that $\varphi_1\geq\varphi_2+C_1$, hence $f^*\varphi_1\geq f^*\varphi_2+C_1$. Therefore, $f^*T_1\preceq f^*T_2$. Conversely, if $f^*T_1\preceq f^*T_2$, then there exists a constant $C_2$ such that $f^*\varphi_1\geq f^*\varphi_2+C_2$, hence $\varphi_1\geq \varphi_2+C_2$ as $f$ is surjective. Thus, $T_1\preceq T_2$.

We now show (b). Assume that $T$ has minimal singularities and let $S_Y\in\mathcal C_{f^*\alpha}$. By Proposition \ref{pro:pullback} there exists a $(1,1)$-current $S_X\in\mathcal C_\alpha$ such that $S_Y=f^*S_X$. Then $T\preceq S_X$, hence $f^*T\preceq f^*S_X=S_Y$ by (a), and consequently $f^*T$ has minimal singularities. Conversely, assume that $f^*T$ has minimal singularities and let $S_X\in\mathcal C_\alpha$. Then $f^*T\preceq f^*S_X$, hence $T\preceq S_X$ by (a), and consequently $T$ has minimal singularities.
\end{proof}

Note that a current with minimal singularities in a pseudoeffective class is unique up to equivalence of singularities, but is in general far from being unique. The following result was shown to us by Tien-Cuong Dinh, and we are grateful to him for allowing us to include it here.

\begin{prop}\label{pro:infinitelymanyminimal}
Let $X$ be a compact complex manifold and let $\alpha$ be a pseudoeffective class in $H^{1,1}_\mathrm{BC}(X,\R)$. If $\alpha$ is not rigid, then there exist infinitely many currents with minimal singularities in $\alpha$.
\end{prop}

\begin{proof}
Let $T$ be a current with minimal singularities in $\alpha$. Since $\alpha$ is not rigid, there exists another closed positive current $S\in\alpha$. We may assume that $S$ is not a current with minimal singularities in $\alpha$, since otherwise for each $t\in[0,1]$ the current $tT+(1-t)S\in\alpha$ has minimal singularities, and the result follows. 

If we fix a smooth $(1,1)$-form $\theta\in\alpha$, then there exist quasi-psh functions $\varphi$ and $\psi$ on $X$ such that
$$T=\theta+dd^c\varphi\quad\text{and}\quad S=\theta+dd^c\psi.$$
For each real number $c$ set
$$\varphi_c:=\max\{\varphi,c+\psi\}\quad\text{and}\quad T_c:=\theta+dd^c\varphi_c.$$
Since $\varphi$ and $c+\psi$ are $\theta$-psh, then so is $\varphi_c$, and hence the closed $(1,1)$-current $T_c$ is positive and clearly $T_c\in\alpha$. As $\varphi\leq\varphi_c$, we have $T_c\preceq T$, hence $T_c$ is also a current with minimal singularities in $\alpha$, for each $c$.

Since $X$ is compact, $\varphi$ and $\psi$ are bounded from above, and we fix a constant $c\gg0$ such that $\varphi_c\neq\varphi$. Assume that $T_c=T$. Then by compactness of $X$ there exists a constant $C$ such that $\varphi_c=\varphi+C$. The choice of $c$ implies that $\varphi_c(x)>\varphi(x)$ for some $x\in X$, hence $C>0$ and consequently $\varphi_c>\varphi$ everywhere on $X$. Therefore, $\varphi_c=c+\psi$ by the definition of $\varphi_c$, and thus $\psi$ has minimal singularities. This contradicts the assumption on the current $S$ made in the first paragraph of the proof. Therefore, $T_c\neq T$, and we conclude as in the first paragraph.
\end{proof}

\begin{rem}\label{rem:rigidvsnumerical}
Let $X$ be a compact complex manifold and let $\alpha$ be a pseudoeffective class in $H^{1,1}_\mathrm{BC}(X,\R)$. Consider the set 
$$|\alpha|_\equiv:=\{[D]\in\alpha\mid D\geq0\text{ is an $\R$-divisor}\}.$$
Then clearly $|\alpha|_\equiv\subseteq\mathcal C_\alpha$. If $\alpha$ is not rigid, then the set $\mathcal C_\alpha$ is actually much bigger than $|\alpha|_\equiv$. We can see this as follows. By Proposition \ref{pro:infinitelymanyminimal} there are infinitely many currents with minimal singularities in $\alpha$. However, \emph{at most one of them} belongs to $|\alpha|_\equiv$: indeed, assume that there exist two distinct effective $\R$-divisors $D_1$ and $D_2$ such that both $[D_1]$ and $[D_2]$ are currents with minimal singularities in $\mathcal C_\alpha$. Then there exists a prime divisor $\Gamma$ on $X$ such that $\mult_\Gamma D_1\neq\mult_\Gamma D_2$. But then $\nu(D_1,\Gamma)\neq\nu(D_2,\Gamma)$ by Thie's theorem, which contradicts \eqref{eq:Tmin}.
\end{rem}

The following lemma is crucial in this paper and it connects non-rigidity with properties of currents with minimal singularities.

\begin{lem}\label{lem:not_rigid}
Let $X$ be a compact complex manifold and let $G\geq0$ be an $\R$-divisor on $X$ such that the class $\{G\}$ is not rigid. Let $G_{\min}$ be a current with minimal singularities in $\{G\}$. Then:
\begin{enumerate}[\normalfont (a)]
\item for each component $\Gamma$ of $G$ we have $\nu(G_{\min},\Gamma)\leq \mult_\Gamma G$,
\item there exists a component $\Gamma$ of $G$ such that $\nu(G_{\min},\Gamma)=0$.
\end{enumerate}
\end{lem}

\begin{proof}
Part (a) follows immediately from the definition of currents with minimal singularities.

For the remainder of the proof we show (b). Let 
\begin{equation}\label{eq:Siumin}
G_{\min}=R+D
\end{equation}
be the Siu decomposition of $G_{\min}$. By (a) we know that
\begin{equation}\label{eq:87}
D\leq G,
\end{equation}
and we first claim that
\begin{equation}\label{eq:claim}
D\neq G.
\end{equation}
Indeed, assume that $D=G$. Pick any $T\in\mathcal C_{\{G\}}$, and let $D_T$ be the divisorial part of the Siu decomposition of $T$. By construction and by \eqref{eq:Tmin} we have 
$$T\geq D_T\geq D=G,$$
and since $T\equiv G$ by assumption, we have that $T-G$ is a closed positive $(1,1)$-current such that $T-G\equiv 0$. Thus $T-G=0$ by Example \ref{exa:zero}, hence $\{G\}$ is a rigid current, a contradiction which shows \eqref{eq:claim}.

By \eqref{eq:87} we may write
$$ G=\sum_{i=1}^r\gamma_i \Gamma_i\quad\text{and}\quad D = \sum_{i=1}^r\delta_i \Gamma_i $$
with $\gamma_i>0$ and $\gamma_i\geq\delta_i\geq0$. Thus, since $G_{\min}\equiv G$, by \eqref{eq:Siumin} we have
\begin{equation}\label{eq:1}
R\equiv \sum_{i=1}^r(\gamma_i-\delta_i) \Gamma_i,
\end{equation}
where clearly
\begin{equation}\label{eq:88}
0\leq\gamma_i-\delta_i\leq\gamma_i\quad\text{for all }i=1,\dots,r. 
\end{equation}
Moreover, by \eqref{eq:claim} there exists an index $1\leq j\leq r$ such that $\gamma_j-\delta_j>0$. This and \eqref{eq:88} show that
$$k:=\max\{t\in\R\mid t(\gamma_i-\delta_i)\leq\gamma_i\text{ for all }i=1,\dots,r\}$$
is a well-defined positive real number. Without loss of generality, we may assume that
\begin{equation}\label{eq:89}
k(\gamma_1-\delta_1)=\gamma_1.
\end{equation}
Set 
$$\Theta:=kR+\sum_{i=1}^r\big(\gamma_i-k(\gamma_i-\delta_i)\big)\Gamma_i.$$
Then $\Theta$ is the sum of two positive closed $(1,1)$-currents, hence is itself a closed positive $(1,1)$-current. We have
\begin{equation}\label{eq:89a}
\nu(\Theta,\Gamma_1)=k\,\nu(R,\Gamma_1)+\big(\gamma_1-k(\gamma_1-\delta_1)\big)=0
\end{equation}
by \eqref{eq:89}.  Moreover, by \eqref{eq:1} we have
$$\Theta\equiv k\sum_{i=1}^r(\gamma_i-\delta_i) \Gamma_i+\sum_{i=1}^r\big(\gamma_i-k(\gamma_i-\delta_i)\big)\Gamma_i=G,$$
thus $\Theta\in\{G\}$. By \eqref{eq:Tmin} and \eqref{eq:89a} this implies 
$$0\leq\nu(G_{\min},\Gamma_1)\leq\nu(\Theta,\Gamma_1)=0,$$
which yields the result.
\end{proof}

\section{Weak rigidity}

In this section we introduce a weaker rigidity property alluded to in Theorems \ref{thm:main1} and \ref{thm:main2}.

\begin{dfn}
    Let $X$ be a compact complex manifold and let $\Gamma\geq0$ be an $\R$-divisor on $X$. We say that $\{\Gamma\}$ is \emph{weakly rigid} if the class of each component of $\Gamma$ is rigid.
\end{dfn}

\begin{rem}\label{rem:rigid_weakrigid}
    Let $X$ be a compact complex manifold and let $G\geq0$ be an $\R$-divisor on $X$. If $\{G\}$ is rigid, then it is weakly rigid by Remark \ref{rem:smallercurrent}. The converse is, however, not true in general. Indeed, consider the birational map as in Remark \ref{rem:pushforward}. Then each component of $\pi^*G$ can be contracted to a smooth projective surface by Castelnuovo's criterion, hence $\{\pi^*G\}$ is weakly rigid by Example \ref{exa:exceptional}. But by Corollary \ref{cor:pullback_rigid}, $\{\pi^*G\}$ is not rigid as $\{G\}$ is not rigid.
\end{rem}

We have the following easy consequence of Corollary \ref{cor:pullback_rigid}.

\begin{cor}\label{cor:pullback_weaklyrigid}
    Let $\pi\colon Y\to X$ be a surjective morphism with connected fibres between compact complex manifolds and let $G$ be an effective $\R$-divisor on $X$. If $\{G\}$ is weakly rigid, then $\{\pi^*G\}$ is weakly rigid.
\end{cor}

\begin{proof}
Assume that $\{G\}$ is weakly rigid. If we write $G=\sum_{i=1}^k \gamma_i \Gamma_i$ as a sum of its components, then all classes $\{\Gamma_i\}$ are rigid. By Corollary \ref{cor:pullback_rigid}, each class $\{\pi^*\Gamma_i\}$ is rigid, hence weakly rigid by Remark \ref{rem:rigid_weakrigid}. Therefore, the class $\{\pi^*G\}= \sum_{i=1}^k \gamma_i \{\pi^*\Gamma_i\}$ is weakly rigid.   
\end{proof}

\begin{rem}
    The converse of Corollary \ref{cor:pullback_weaklyrigid} is not true in general. Indeed, with notation as in Remark \ref{rem:rigid_weakrigid}, the class $\{\pi^*G\}$ is weakly rigid, but $\{G\}$ is not rigid, hence not weakly rigid by Remark \ref{rem:rigid_weakrigid}.
\end{rem}

\begin{exa}
Let $D$ be an effective $\R$-divisor on a complex projective surface $S$. Then there always exists a birational morphism $f\colon S'\to S$ such that $\{f^*D\}$ is weakly rigid. Indeed, by passing to a resolution we may assume that $S$ is smooth. If $\Gamma$ is a prime divisor on $S$, and if $\Gamma'$ is the strict transform of $\Gamma$ on the blowup of $S$ at a point lying on $\Gamma$, then $(\Gamma')^2=\Gamma^2-1$. Therefore, if $f$ is obtained by blowing up sufficiently many points on the components of $D$, then we may assume that each component of $f^*D$ has negative self-intersection. Then the statement follows from \cite[Proposition 3.13 and Theorem 4.5]{Bou04}. 
\end{exa}

The following lemma will be crucial in the proofs of Theorems \ref{thm:main1} and \ref{thm:main2}.

\begin{lem}\label{lem:lelong_num_notwrigid}
Let $X$ be a compact complex manifold and let $G\geq0$ be an $\R$-divisor on $X$. Assume that the class $\{G\}$ is not weakly rigid and let $\Gamma$ be an irreducible component of $G$ such that $\{\Gamma\}$ is not rigid. Let $G_{\min}\in\{G\}$ and $\Gamma_{\min}\in\{\Gamma\}$ be currents with minimal singularities. Then
$$ \nu(\Gamma_{\min},\Gamma)=0 \quad \text{ and } \quad \nu(G_{\min},\Gamma)=0. $$
\end{lem}

\begin{proof}
Since $\{\Gamma\}$ is not rigid, we have $\nu(\Gamma_{\min},\Gamma)=0$ by Lemma \ref{lem:not_rigid}. For the other statement, we may write 
$$ G = \gamma \Gamma + G', $$
where $\gamma>0$ and $G'$ is an effective $\R$-divisor on $X$ such that $\Gamma$ is not a component of $G'$. Note that $\gamma \Gamma_{\min} + G'\in\{G\}$, and thus, by the definition of currents with minimal singularities, we have
$$ 0\leq \nu(G_{\min},\Gamma) \leq \nu(\gamma \Gamma_{\min} + G',\Gamma) = \gamma\,\nu(\Gamma_{\min},\Gamma) + \nu(G',\Gamma) = 0, $$
which finishes the proof.
\end{proof}

\section{Proofs of the main results}\label{sec:pf}

Finally, in this section we prove the main results of the paper.

\begin{proof}[Proof of Theorem \ref{thm:main1}]
\emph{Step 1a.}
We may write
$$ K_Y+\Delta_Y \sim_\R f^*(K_X+\Delta) + E \sim_\R f^*D + E, $$
where $E\geq0$ is an $f$-exceptional $\R$-divisor having no common components with $\Delta_Y$. Then by Corollary \ref{cor:pullback_weaklyrigid} it suffices to show that the class $\{f^*D\}$ is weakly rigid, hence it suffices to show that $\{f^*D+E\}$ is weakly rigid. 

By replacing $(X,\Delta)$ by $(Y,\Delta_Y)$ and $D$ by $f^*D+E$, we may assume from the start that the pair $(X,\Delta+D)$ is log smooth. We will show that the class $\{D\}$ is weakly rigid.

\medskip

\emph{Step 1b.}
Fix $0<\varepsilon\ll1$ such that the pair $(X,\Delta+\varepsilon D)$ is klt. Then
$$ K_X+\Delta+\varepsilon D \sim_\R (1+\varepsilon)D. $$
By replacing $(X,\Delta)$ by $(X,\Delta+\varepsilon D)$ and $D$ by $(1+\varepsilon)D$, we may assume from the start that
\begin{equation}\label{eq:support}
\Supp D\subseteq\Supp \Delta.
\end{equation}

\medskip

\emph{Step 1c.}
Let $\Delta''$ be a $\Q$-divisor such that $\Delta''\geq\Delta$ and the pair $(X,\Delta'')$ is klt. Then
\begin{equation}\label{eq:9}
K_X+\Delta''\sim_\R D+(\Delta''-\Delta).
\end{equation}
If $\kappa(X,K_X+\Delta'')>0$, then $(X,\Delta'')$ has a good model by \cite[Theorem 1.3]{Laz24}, hence $(X,\Delta)$ has a good model by \cite[Theorem 1.4(b)]{Laz24}, in which case we conclude by Lemma \ref{lem:rigidMMP} and by Remark \ref{rem:rigid_weakrigid}.

Therefore, we may assume that $\kappa(X,K_X+\Delta'')=0$. Then there exists a $\Q$-divisor $D''\geq0$ such that $K_X+\Delta''\sim_\Q D''$. By \cite[Lemma 2.5(b)]{Laz24}, this and \eqref{eq:9} imply that $D''=D+(\Delta''-\Delta)$. Hence, it suffices to show that the divisor $D''$ is weakly rigid. By replacing $(X,\Delta)$ by $(X,\Delta'')$ and $D$ by $D''$, we may assume from the start that 
$$\Delta\text{ and }D\text{ are }\Q\text{-divisors}$$
and
\begin{equation}\label{eq3a}
\kappa(X,K_X+\Delta)=0.
\end{equation}

\medskip

\emph{Step 2.}
Suppose by contradiction that the class $\{D\}$ is not weakly rigid, and let $\Gamma$ be a component of $D$ such that $\{\Gamma\}$ is not rigid. Then $\Gamma$ is a component of $\Delta$ by \eqref{eq:support}, hence we may write
$$ \Delta = \gamma \Gamma + \Delta', $$
where $ \Gamma $ is not a component of $\Delta'$ and $0<\gamma<1$. Set $D':=D + (1-\gamma) \Gamma$, and observe that $D'$ is not weakly rigid and
\begin{equation}\label{eq1}
K_X + \Gamma + \Delta' \sim_\Q D'.
\end{equation}
The pair $(X,\Gamma+\Delta')$ is plt, and by \eqref{eq:support} we have
\begin{equation}\label{eq:inclusions}
\Gamma \subseteq \Supp D'\subseteq \Supp(\Gamma+\Delta').
\end{equation}
Moreover, since $\Supp D'=\Supp D$, by \cite[Lemma 2.9]{DL15} and by \eqref{eq3a} and \eqref{eq1} we have
\begin{equation}\label{eq3}
\kappa(X,K_X+\Gamma+\Delta')=0.
\end{equation}

Let $D_{\min}\in\{D\}$ and $\Gamma_{\min}\in\{\Gamma\}$ be currents with minimal singularities. Note that $D_{\min}+(1-\gamma) \Gamma_{\min}\in\{K_X + \Gamma + \Delta'\}$ by \eqref{eq1}, and thus, by Lemma \ref{lem:lelong_num_notwrigid} we have
$$0\leq\nu\big(D_{\min}+(1-\gamma) \Gamma_{\min}, \Gamma\big) = \nu(D_{\min},\Gamma)+(1-\gamma) \nu(\Gamma_{\min}, \Gamma) =0.$$
Therefore, by Lemma \ref{lem:sigmanu} we infer
\begin{equation}\label{eq2}
\sigma_\Gamma(K_X + \Gamma + \Delta')=0.
\end{equation}

\medskip

\emph{Step 3.}
From \eqref{eq3} we deduce that, in particular, the pair $(X,\Gamma+\Delta')$ has an NQC weak Zariski decomposition, see \cite[\S2.3]{LT22}. Therefore, by our assumption in dimension $n-1$ and by \cite[Theorem F]{LT22} we may run a $(K_X + \Gamma + \Delta')$-MMP which terminates with a minimal model $ (Z,\Gamma_Z+ \Delta'_Z)$, where $\Gamma_Z$ and $\Delta'_Z$ are the strict transforms of $\Gamma$ and $\Delta'$, respectively. By \eqref{eq2} and by Lemma \ref{lem:sigmaMMP} we have that $\Gamma_Z\neq0$. Note that 
$$\text{the pair }(Z,\Gamma_Z+ \Delta'_Z)\text{ is plt},$$
and we have 
\begin{equation}\label{eq:4}
\kappa(Z,K_Z+\Gamma_Z+ \Delta'_Z)=0
\end{equation}
by \eqref{eq3}.

Let $D'_Z$ be the strict transform of $D'$ on $Z$. Then \eqref{eq1} gives
\begin{equation}\label{eq:2}
K_Z + \Gamma_Z + \Delta'_Z \sim_\Q D'_Z,
\end{equation}
and by \eqref{eq:inclusions} we have
\begin{equation}\label{eq:3}
\Gamma_Z \subseteq \Supp D'_Z \subseteq \Supp(\Gamma_Z + \Delta'_Z).
\end{equation}
Our assumption in dimension $n-1$ implies the Nonvanishing conjecture in dimension $n-1$, hence by \cite[Corollary 1.8]{DHP13} we obtain that $\Gamma_Z$ does not belong to the stable base locus of $K_Z + \Gamma_Z + \Delta'_Z$, which is equal to $\Supp D_Z'$ by \eqref{eq:4} and \eqref{eq:2}. This contradicts \eqref{eq:3} and concludes the proof.
\end{proof}

\begin{proof}[Proof of Theorem \ref{thm:main2}]
As in Step 1a of the proof of Theorem \ref{thm:main1}, we may assume that the pair $(X,\Delta+D)$ is log smooth. In particular, the pair $(X,\Delta)$ is then dlt. We need to show that the class $\{D\}$ is weakly rigid.

If the $\R$-divisor $K_X+\Delta-\varepsilon\lfloor\Delta\rfloor$ is not pseudoeffective for each positive real number $\varepsilon$, then the pair $(X,\Delta)$ has a good model by \cite[Theorem 3.1(a)]{Laz23}, and thus $D$ is even rigid by Lemma \ref{lem:rigidMMP}.

Thus, we may assume that there exists $\delta'>0$ such that the $\R$-divisor $K_X+\Delta-\delta'\lfloor\Delta\rfloor$ is pseudoeffective. Since we assume the Nonvanishing conjecture in dimension $n$, there exists an effective $\R$-divisor $D_{\delta'}$ such that 
$$K_X+\Delta-\delta'\lfloor\Delta\rfloor\sim_\R D_{\delta'}.$$
Pick a real number $0<\delta<\delta'$. Then
\begin{equation}\label{eq:5}
K_X+\Delta-\delta\lfloor\Delta\rfloor\sim_\R D_{\delta'}+(\delta'-\delta)\lfloor\Delta\rfloor
\end{equation}
and
\begin{equation}\label{eq:6}
K_X+\Delta\sim_\R D_{\delta'}+\delta'\lfloor\Delta\rfloor.
\end{equation}
Since $\kappa_\iota(X,K_X+\Delta )=0$ and $K_X+\Delta\sim_\R D$, this together with \eqref{eq:6} implies by \cite[Lemma 2.5(b)]{Laz24} that
$$D=D_{\delta'}+\delta'\lfloor\Delta\rfloor.$$
This, in particular, yields
\begin{equation}\label{eq:9a}
\Supp D=\Supp(D_{\delta'}+(\delta'-\delta)\lfloor\Delta\rfloor),
\end{equation}
hence
$$\kappa_\iota(X,K_X+\Delta-\delta\lfloor\Delta\rfloor)=0$$
by \eqref{eq:5} and by \cite[Remark 2.4(b)]{Laz24}. Since the pair $(X,\Delta-\delta\lfloor\Delta\rfloor)$ is klt, we conclude that the class of the $\R$-divisor $D_{\delta'}+(\delta'-\delta)\lfloor\Delta\rfloor$ is weakly rigid by \eqref{eq:5} and by Theorem \ref{thm:main1}. Therefore, the class $\{D\}$ is weakly rigid by \eqref{eq:9a}.
\end{proof}

	%Bibliography
	\bibliographystyle{amsalpha}
	\bibliography{biblio}
	
\end{document}